\documentclass{amsart}
\usepackage{amssymb}
\usepackage{amsmath}
\theoremstyle{plain}
\theoremstyle{plain}
\newtheorem{lemma}{Lemma}[section]
\newtheorem{thrm}{Theorem}[section]
\newtheorem{propn}{Proposition}[section]

\theoremstyle{definition}

\newtheorem{define}{Definition}[section]

\newtheorem{remk}{Remark}[section]
\newtheorem{exmp}{Example}[section]
\newtheorem{algthm}{Algorithm}[section]

\newcommand{\bt}{\widetilde{\mathfrak{b}}}

\DeclareMathOperator*{\aut}{GA} \DeclareMathOperator*{\aff}{Af}
\DeclareMathOperator*{\Gl}{GL} \DeclareMathOperator*{\El}{EA}
\DeclareMathOperator*{\BA}{BA} \DeclareMathOperator*{\tame}{T}
 \DeclareMathOperator*{\VA}{VA}
\DeclareMathOperator*{\TV}{TV} 
 \DeclareMathOperator{\tameequiv}{\ensuremath{\sim}}
 \DeclareMathOperator{\stameequiv}{\displaystyle \sim_{\bf{st}}}

\begin{document}
\thanks{This is part of the author's doctoral thesis, written at Washington University under the direction of David Wright}
\title{Some Stably Tame Polynomial Automorphisms}
\author{Sooraj Kuttykrishnan}
\address{Department of Computer Science, Washington University in St. Louis, MO-63112, USA}
\email{sooraj@cse.wustl.edu}
\begin{abstract}
We study the structure of length three polynomial automorphisms of $R[X,Y]$ when $R$ is a UFD. These results are used to prove that if $\text{SL}_m(R[X_1,X_2,\ldots, X_n]) = \text{E}_m(R[X_1,X_2,\ldots, X_n])$ for all $n,\ge 0$ and for all $m \ge 3$ then all length three polynomial automorphisms of $R[X,Y]$ are stably tame.
\end{abstract}
\maketitle
\section{Introducton}

Unless otherwise specified $R$ will be a commutative ring with 1 and
$R^{[n]}=R[X]=R[X_1,..., X_n]$ is the polynomial ring in $n$
variables. A polynomial map is a map $F=(F_1,..., F_n):
\mathbb{A}^n_R \rightarrow \mathbb{A}^n_R$ where each $F_i \in
R^{[n]}$. Such an $F$ is said to be invertible if there exists $G =
(G_1,..., G_n), G_i \in R^{[n]}$ such that $G_i(F_1,..., F_n)=X_i$
for $ 1\le i\le n$. Invertible polynomial maps are in one to one
correspondence with R-automorphisms of the polynomial ring $R^{[n]}$
via the map $ F \rightarrow F^*, F^*(g)= g(F),\ g \in R^{[n]}$. So
we identify the group of R-automorphisms of $R^{[n]}$ with the group
of all invertible polynomial maps in n variables. Notice that this
identification is not an isomorphism but rather an anti isomorphism. We would like to understand the structure of
 \begin{itemize}
\item $\aut_n(R)= \{F=(F_1,\ldots, F_n): F$ is invertible \}.\\
\end{itemize}
Some subgroups of $\aut_n(R)$ are the following.
\begin{itemize}
\item
The affine subgroup:
$\aff_n(R)=\{(a_{11}X_1+a_{12}X_2+\ldots +a_{1n}X_n+b_1,\ldots,a_{n1}X_1+..a_{nn}X_n+b_n):
(a_{ij}) \in \Gl_n(R) \:{\mbox{ and }}\:b_i \in R\}$\\

\item The elementary subgroup: $\El_n(R)=$ The
subgroup generated by automorphisms of the form
$(X_1,X_2,\ldots ,X_{i-1}, X_i+f(X_1,\ldots,X_{i-1},\hat{X_i},
X_{i+1},\ldots , X_n),\\
\ldots ,X_n)
 \text{ where }f\in R[X_1,X_2,\ldots , \hat{X_i},\ldots, X_n],\ i \in \{1,\ldots, n\}$.\\

\item The triangular subgroup: $\BA_n(R)$= The subgroup
 of all R-automorphisms of the form $F=(a_1X_1+f_1(X_2,\ldots , X_n), a_2X_2+f_2(X_3,\ldots, X_n), \ldots , a_nX_n+f_n)$
 where each $a_i \in R^*$ and $f_i \in R[X_{i+1},\ldots ,X_n]$ for all $1\le i \le n-1$ and $f_n \in
 R$.\\

 \item Tame subgroup: $\tame_n(R)=\langle \aff_n(R), \El_n(R)\rangle$.
 \end{itemize}

It is easy to see that $\aut_1(R)={\aff_1}(R)$ when $R$ is a domain. The structure of $\aut_2(R)$ when $R$ is a field $k$ is well known and is the so-called Jung-van der
Kulk theorem or the Automorphism
  Theorem.\cite{Jung}, \cite{vanderKulk}
\begin{thrm}(Jung, van der Kulk)
If k is a field then $\aut_2(k)= \tame_2(k)$. Further, $\tame_2(k)$
is the amalgamated free product of $\aff_2(k)$ and $\BA_2(k)$ over
their intersection.
\end{thrm}
However, not much is known about $\aut_3(k)$. A natural question is whether $\tame_3(k)$ the whole group $\aut_3(k)$? Nagata
\cite{Nagata} conjectured that the answer is no and gave a candidate
counterexample.
\begin{exmp}\label{Nagatasexample}(Nagata)
$$N=(X+t(tY+X^2),Y-2(tY+X^2)X-t(tY+X^2)^2,t) \in \aut{_3}(k)$$
\end{exmp}
Let $R$ be a domain. Then the following algorithm from \cite{Arnosbook} will determine if $F=(P(X,Y),Q(X,Y))\in \aut_2(R)$ is in $\tame_2(R)$. Let $tdeg(F)=deg(P)+deg(Q)$ and $h_1$ be the highest degree term of $P$ and $h_2$ that of $Q$.
\begin{algthm}\label{algorithmtotesttameness}
Input: $F=(P,Q)$.\\
1) Let $(d_1,d_2)=(deg(P),deg(Q))$.\\
2) If $d_1=d_2=1$, go to 7.\\
3) If $d_1\neq d_2$, go to 5.\\
4) If there exists $\tau \in \aff_2(R)$ with $tdeg(\tau\circ F) < tdeg(F)$, replace $F$ by $\tau \circ F$ and go \indent to 1, else stop : $\notin \tame_2(R)$.\\
5) If $d_2<d_1$, replace $F$ by $(Q,P)$.\\
6) If $d_1\mid d_2$ and there exists $c\in R$ with $h_2= ch_1^{d_2/d_1}$, replace $F$ by $(X,Y-cX^{d_2/d_1})\circ F$ \indent and go to 1, else stop : $F \notin \tame_2(R)$.\\
7) If $\det JF \in R^*$, stop: $F\in \tame_2(R)$, else stop : $F\notin \tame_2(R)$.
\end{algthm}
Using this algorithm we can easily conclude that $N\notin \tame_2(k[t])$. We say that $N$ is $`t'$ wild.
Shestakov and Umirbaev in 2002 \cite{NagataisWild} proved that $N
\notin \tame_3(k)$ and thus proved Nagata's conjecture. \\
We can extend $N$ from the Example \ref{Nagatasexample} naturally as $\widetilde{N}=(N,W) \in \aut_4(k)$. Martha Smith proved \cite{MarthaSmith} that $\widetilde{N}\in \tame_4(k)$.
\begin{define}
Let $ F,G \in \aut_n(R)$. Then
\begin{enumerate}
              \item $F$ is \emph{stably tame} if there exists $m \in {\mathbb
{N}}$ and new variables $ X_{n+1},\ldots, X_{n+m}$ such that the
extended map $\widetilde{F}=(F, X_{n+1},\ldots, X_{n+m})$ is tame.\\
 i.e $(F, X_{n+1},\ldots, X_{n+m})\in \tame_{n+m}(R)$

 \item  $F$ is \emph{tamely equivalent}${(\tameequiv)}$ to $G$ if there exists $H_1, H_2 \in
\tame_n(R)$ such that $ H_1\circ F \circ H_2 = G$.

\item $F$ is \emph{stable tamely equivalent}$(\stameequiv)$ to $H
\in \aut_{n+m}(R)$ if there exists $ \widetilde{H_1}, \widetilde{H_2
}\in \tame_{n+m}(R)$ such that $\widetilde{H_1} \circ \widetilde{F}
\circ \widetilde{H_2}=H$ where $\widetilde{F}=(F, X_{n+1},\ldots ,
X_{n+m})$
 \end{enumerate}
\end{define}

So $N$ from Nagata's example is stably tame with one more variable. Also, $N$ fixes `$t$' and so $N \in
\aut_2(k[t]).$ Viewed this way, by the automorphism theorem $N$ is a
tame k(t)-automorphism. In fact this phenomenon occurs in a more
general situation as described in the next section.
\section{Length Of An Automorphism}
\begin{propn}
Let $R$ be a domain $K$ its fraction field and $F\in \aut_2(R)$.
Then $F=L\circ D_{a,1}\circ F_{m}\circ F_{m-1}\circ...\circ F_1$
where $L=(X+c, Y+d),\ D_{a,1}=(aX,Y),\ F_i=(X,Y+f(X))\text{ or }
F_i=(X+g(Y),Y)$ for some $c,d \in R,\ a \in R^*,\ f(X),\ g(X) \in
K[X] $
\end{propn}
\begin{proof}
 Let $F=(P(X,Y),Q(X,Y)), \text{ where } P(X,Y),\ Q(X,Y) \in R[X,Y]$ and $ L=(X+c,Y+d),\text{ with }\ c=P(0,0)\text{ and }d=Q(0,0)$. Let $G=L^{-1}\circ F\in \aut_2^0(R)$.
   Viewed as an element of $\aut_2^0(K)$, by the Automorphism Theorem
   $G\in \tame_2^0(K)$. When R is a domain, by the results of Wright \cite{Amalgamatedfreeproduct},
   the group $\tame_2^0(K)$ of tame automorphisms of $K[X,Y]$ preserving the
augmentation has a similar description as a free amalgamated product
as $\aut_2(k)$ where k is any field. In particular, $\tame_2^0(K)$
is generated by the automorphisms \vspace{-.75pc}

$$F_1=(X,Y+f(X)),\ F_2=(X+g(Y),Y),\
D_{a ,b }=(a X,b Y)$$ where $f(X)\in K[X],\ g(Y)\in R[Y],f(0)=g(0)=0
,a ,b\in K^{\ast }$. Since $D_{a ,b }=D_{ab,1}\circ D_{b ^{-1},b }$
and $SL_{2}(K)=E_{2}(K)$ we have that $D_{b^{-1},b }$ is
 a product of elementary linear automorphisms and hence we can assume that
 $b=1$. We also have the following equalities.
 \begin{align*}
 F_1\circ D_{a,1}&=(aX,Y+f(aX))=D_{a,1}\circ F_1^\prime\text{ where }F_1^\prime
=(X,Y+f(aX)). \\
 F_2\circ D_{a,1}&=(aX+g(Y),Y)=D_{a,1}\circ F_2^\prime \text{ where }
F_2^\prime =(X+a^{-1}g(Y),Y).\end{align*}

So if $G \in \aut_2^0(R)$ then $G= D_{a,1}\circ F_{m} \circ F_{m-1}
\circ \ldots \circ F_2 \circ F_1$ where each $F_i$ is either of the type
$(X, Y+f_i(X))$ or $ (X+g_i(Y),Y), f_i(X),g_i(X)\in K[X] \text{ and
} a\in K^*$ The linear components of $G$ and $F_{m} \circ F_{m-1}
\circ \ldots\circ F_2 \circ F_1$ are in $GL_{2}(R)$ and $SL_{2}(K)$,
respectively. This  implies that $a \in R^{\ast }$ and both
$D_{a,1},\ F_{m} \circ F_{m-1} \circ...\circ F_2 \circ F_1 \in
\aut_2^0(R)$.
\end{proof}
\begin{define}\
\begin{enumerate}
\item \emph{Length} of $F \in \aut_2^0(R)$ is the smallest natural number m such that
 $F=D_{a,1}\circ F_m\circ F_{m-1} \circ \ldots \circ F_2 \circ F_1$ where each $F_i$ is either
of the type $(X, Y+f_i(X))$ or $ (X+g_i(Y),Y)$ with $f_i(X),\ g_i(X)
\in K[X], a\in R^*$ and $f_i(0)=g_i(0)=0$.
\item $\text{L}^{(m)}(R)=\{F \in {\aut}_2^0(R):\ F \text { is of length } m\}$
\end{enumerate}
\end{define}
\begin{remk}\label{Da1doesn't matter}If $F \in \text{L}^{(m)}(R)$ as above and $F=D_{a,1}\circ F_m\circ F_{m-1} \circ \ldots \circ F_2 \circ F_1 \in \text{L}^{(m)}(R)$ then $F$ is tamely equivalent to $G=F_m\circ F_{m-1} \circ \ldots \circ F_2 \circ F_1$. Thus $F$ is stably tame iff $G$ is stably tame.
\end{remk}
Clearly if $F\in \text{L}^{(1)}(R)$ then $F \in \tame_2(R)$. Suppose $F \in \text{L}^{(2)}(R)$. Then $F= D_{a,1}\circ F_2\circ F_1 \text{ with } F_1=(X, Y+f_1(X))$ and $ F_2=(X+g(Y),Y)$ as in the definition above. $ G=D_{a,1}^{-1}\circ F=(X+g(Y+f(X)),Y+f(X))\in \aut_2(R) \Rightarrow f(X) \in R[X]$. Putting $X=0$ in the first coordinate of $G$ we get that $g(Y) \in R[Y]$. So $F \in \tame_2(R)$. Thus the first non trivial case is of length three.

Now lets go back to Nagata's example.
\begin{align*}
\mbox{ Let } \:F_1&=(X,Y+{{\displaystyle{\displaystyle X^2\over t}}},t) \text{ and } F_2=(X+t^2Y,Y)\\
\mbox{ Then }\: N &=F_1^{-1}\circ F_2 \circ F_1.
\end{align*}
So Nagata's example is of length three and it is stably tame with one
more variable. Drensky and Yu \cite{TameAndWild} began a systematic
study of length three automorphisms and proved the following result.
\begin{thrm}(Drensky, Yu) Let k be a field of characteristic zero and $ F \in \text{ L }^{(3)}(k[t])$ such that
 $F = F_1^{-1}\circ G \circ F_1$ where $F_1=(X,Y+f(X)), G=(X+g(Y),Y)$ with $f(X),\ g(X) \in k[t][X].$
  Then $F$ is stably tame with one more variable.
\end{thrm}
\section{Stable Tameness Of Polynomials}
  Another important notion is the stable tameness of polynomials. This was studied by Berson in \cite{Bersonclassdefn}, Edo and V\'en\'ereau in \cite{EdoAndVenereauOnVariables} and Edo in \cite{EDOtotallystablytame}. We'll give some relevant results from these papers below.\\
   Let A be any commutative ring with 1. A polynomial $P(X)\in A^{[n]}$ is said to be a variable if there
 exists $F \in \aut_n(A)$ such that $F=(F_1,F_2,\ldots F_n) \text{ and } F_1(X)=P(X)$.
\begin{define}\
\begin{align*}
  {\VA}_n(A) =\{P \in A^{[n]}:&\mbox{ There exists}\ F \in \aut_n(A) F=(F_1,F_2,\ldots F_n) \text{ and }\\ &F_1(X)=P(X).\}\\
  {\TV}_n(A)=\{P \in A^{[n]}:&\mbox{ There } \mbox{ exists }\: F \in
\tame_n(A) F=(F_1,F_2,\ldots F_n) \text{ and }\\
 &F_1(X)=P(X).\}
\end{align*}
\end{define}
Following definition is due to Berson \cite{Bersonclassdefn}.
\begin{define}(Berson's Class)
 $l \in \mathbb{N},p_0 \in
 A^*,  g_0,p_1,...p_l \in A\mbox{ and }Q_1,..., Q_l \in A^{[1]},$ we define
 $P_l \in A^{[2]}$ by induction on $l$.
\begin{align*}
 P_0=&p_0X+g_0,\\
 P_1=&p_1Y+Q_1(X),\\
 P_2=&p_2X+Q_2(p_1Y+Q_1(X)),\\
 P_l=&p_lP_{l-2}+Q_l(P_{l-1})\:\: \mbox {for} \:\: l\ge 3.\\
 \mathcal{B}^l(A)&=\{P_l: p_0 \in A^*, g_0, p_1,\ldots , p_l \in A,Q_1,\ldots , Q_l \in
 A^{[1]}\}\\
 \mathcal{B}(A)&=\bigcup_{l \in \mathbb{N}} \mathcal{B}^l(A)\quad \mbox{(Berson's polynomials)}\\
 \mathcal{B}V_2(A)&={\VA}_2(A)\cap \mathcal{B}(A)\quad \mbox{(Berson's variables)}\\
 \mathcal{B}V_2^l(A)&={\VA}_2(A)\cap \bigcup_{i\le\:
 l}\mathcal{B}^i(A)\\
 \end{align*}
\end{define}
 \begin{define}\
 \begin{enumerate}
       \item (Stably tame polynomial) A polynomial $P\in R^{[n]}$ is stably tame if
 there exists $F \in \tame_{n+m}(R),\: m\ge 0$ such that $F=(F_1,F_2,\ldots F_n) \text{ and } F_1(X)=P(X)$
       \item (Totally stably tame polynomial)(Edo,\ \cite{EDOtotallystablytame})
       A polynomial $P\in R^{[n]}$ is totally stably tame if
 there exists a stably tame automorphism $F\in {\aut}_n(R)$ such
 that $F=(F_1,F_2,\ldots F_n) \text{ and } F_1(X)=P(X)$.
     \end{enumerate}
 \end{define}
 Following theorem is claimed by Eric Edo \cite{EDOtotallystablytame}. However, it appears that additional hypothesis are required in his proof.
\begin{thrm}\label{edo'stheorem}
 If $ F\in \mathcal{B}V_2^2(R)$ where R is a UFD then $F$ is
totally stably tame
\end{thrm}
 \begin{remk}
 \
 \begin{enumerate}
                \item If $F\in R^{[n]}$ is totally stably tame then
                it is stably tame.
                 \item If $P \in R^{[2]}$ is a totally stably tame
 polynomial and $F \in \aut_2(R)$ be such that $F(X_1)=P$ then
 $F$ is a stably tame automorphism.
                                \end{enumerate}
 \end{remk}
 \section{Main Theorem And Structure Of Length Three Automorphisms}
 Let $\text{SL}_n(R)$ denote the set of all $n\times n$ matrices with entries from R and determinant equal to 1 and $\text{E}_n(R)$ denote the group generated by the set of all nxn elementary matrices with entries from $R$.

\begin{thrm}(Main Theorem)\label{length3isstablytame1}\ Suppose $R$ is a UFD such that
\begin{equation*}
\text{SL}_m(R[X_1,X_2,\ldots, X_n]) = \text{E}_m(R[X_1,X_2,\ldots, X_n])
  \end{equation*}
for all $n\ge 0$ and for all $m\ge 3$.
 Then $F\in \text{L}^{(3)}(R) \Rightarrow F$ is stably tame.
\end{thrm}
Eric Edo claimed this result in \cite{EDOtotallystablytame} (Theorem 7) without the assumption that $\text{SL}_m(R[X_1,X_2,\ldots, X_n]) = \text{E}_m(R[X_1,X_2,\ldots, X_n])$ for all $n\ge 0$ and for all $m\ge 3$. A brief outline of his proof is as follows. If $F =(F_1,F_2) \in \text{L}^{(3)}(R)$ then $F_1$ has the form $qX+H(pY+G(X)),\ q,p\in R,\ H,G \in R[X]$. If $ht(p)=0$, then $F$ is tame. The next step is to show that $F\stameequiv F^1 (F_1^1,F_2^1) \in \text{L}^{(3)}(R[X])$ where $F_1^1$ has the form $q^1X+H^1(p^1Y+G^1(X)),\ q^1,p^1\in R[X],\ H,G \in R[X][W]$ with $ht(p^1)<ht(p)$ and then we are done by induction on $ht(p)$. However this step involves composing $F$ with an affine map $a_3(R) \in \aff_2(R)$. At the next step of the induction such a map will be in $\aff_2(R[X])$ and hence not necessarily in $\tame_3(R)$. So we believe that the assumption that $\text{SL}_m(R[X_1,X_2,\ldots, X_n]) = \text{E}_m(R[X_1,X_2,\ldots, X_n])$ for all $n\ge 0$ and for all $m\ge 3$ is required. Also, our methods are quite different from his.

\begin{remk}$\text{SL}_m(R[X_1,X_2,\ldots, X_n]) = \text{E}_m(R[X_1,X_2,\ldots, X_n])$ for all $n\ge 0$ and for all $m\ge 3$ if $R$ is a regular ring.
\end{remk}
\begin{remk}In \cite{berson-2007} Berson,van den Essen and Wright recently proved that if $F \in \aut_2(R)$, where $R$ is a regular ring then $F$ is stably tame. This is a much stronger result. However, our result does not require the ring to be regular.
\end{remk}
We will give two different proofs of Theorem \ref{length3isstablytame1}. First proof will use Theorem \ref{edo'stheorem}. The second proof is different, self contained and will use the hypothesis that $\text{SL}_2(R[X_1,X_2,\ldots, X_n]) = \text{E}_2(R[X_1,X_2,\ldots, X_n])$ for all $n$.
However, before proving these theorems, we would like to know if there are examples of length three automorphisms that are not covered by Drensky and Yu's theorem \cite{TameAndWild}. i.e
Does $F \in \text{L}^{(3)}(R) \Rightarrow F= F_1^{-1}\circ G_1 \circ F_1$? Automorphisms of this kind are called conjugates.
 The answer is no and here is an example due to Wright\cite{Wrightprivate}.
\begin{exmp}\label{wright's example}
\
  Let $t\in R\backslash\{0\}$ and $F= F_2\circ G_1 \circ F_1$ where \\
$F_1 =(X, Y+{\displaystyle X^2\over \displaystyle t^2}),\ G_1=(X+t^3Y, Y) \text{ and }F_2=(X, Y-{{\displaystyle {\displaystyle X^2\over t^2}+{2X^3\over t}}})$.\\

\noindent Then $F=(X+t(t^2Y+X^2),\ Y-(t^2Y+X^2)^2-2tYX+t^2(t^2Y+X^2)^3+$\\
 \indent \hspace{8cm}$3X^2(t^2Y+X^2)+3tX(t^2Y+X^2)^2)$.
\end{exmp}

Following \cite{TameAndWild} we prove the below lemma.
\begin{lemma}\label{length3structurelemma1} Let $F\in L^{(3)}(R)$, and  $F =F_2\circ G_1 \circ F_1$ where $F_i=(X,Y+f_i(X)), G_1 =(X+g(Y),Y),
f_i \in K[X], g \in K[Y], f_i(0)=g(0)=0.$ Then
$f_i=\frac{\displaystyle A_i(X)}{\displaystyle b}$ and $g= D(bY)$
where $A_i(X) \in R[X],D(Y) \in R[Y] , b \in R $ and $b$ and $A_i$
do not have any common factors in $R[X]$.
\end{lemma}
\begin{proof}
 We rewrite  $f_i = \frac{\displaystyle A_i(X)} {\displaystyle b_i}, $  where $A_i(X) \in R[X], A_i(0)=0, b_i \in R $
  and $A_i(X)$ and $b_i$ has no common factors in $R[X]$.\\
Since $F=(X+g(Y+f_1(X), Y+f_1(X)+f_2(X+g(Y+f_1(X)))) \in
\aut_2^0(R),$
 \begin{equation}\label{1.1}
 g(Y+f_1(X))=g\left (Y+{\frac{A_1(X)}{b_1}}\right )=
\sum _{i=0}^{n} {\frac{g^{(i)}(Y)A_1^{i}(X)}{i!b_1^{i}}}\in R[X,Y].
\end{equation}
Putting $X=0$ in \eqref{1.1} we get  $g(Y)\in R[Y] \Rightarrow
\displaystyle{\sum _{i=1}^{n}
{\frac{g^{(i)}(Y)A_1^{i}(X)}{i!b_1^{i}}}\in R[X,Y]}$. So,
$$A_1(X)\left(\displaystyle{{g^\prime(Y) b_1^{n-1}\over 1!}+ {g^{\prime
\prime}(Y)A_1(X) b_1^{n-2}\over 2!}+...+
{g^{(n)}(Y)A_1(X)^{n-1}\over n!}}\right )\equiv 0 \left ( {\text{
mod }b_1^n}\right ).$$ Since $A_1(X)$ and $b_1$ does not have a
common factor we get,
\begin{equation}\label{1.2}
\left(\displaystyle{{g^\prime(Y) b_1^{n-1}\over 1!}+ {g^{\prime
\prime}(Y)A_1(X) b_1^{n-2}\over 2!}+
 ...+ {g^{(n)}(Y)A_1(X)^{n-1}\over n!}}\right )\equiv 0 \left ( {\text{ mod }b_1^n}\right )
 \end{equation}
Putting $X=0$ in \eqref{1.2} we get,
 \begin{equation*}
 g^\prime(Y) {b_1^{n-1}} \equiv 0 \left (\text{ mod }b_1^n\right )\\
 \Rightarrow g^\prime(Y) \equiv 0 \left (\text{ mod }b_1\right )
\end{equation*}
Hence the coefficient of $Y^i$ in $g(Y)$ is divisible by $b_1$ for
$i \ge 1$. Let $g^{\prime}(Y)=b_1g_1(Y)$ for some $g_1(Y)\in R[Y]$.
So \eqref{1.2} becomes
\begin{align*}
\left(\displaystyle{{g_1(Y) b_1^{n}\over 1!}+ {g_1^{\prime}(Y)A_1(X)
b_1^{n-1} \over 2!}+...+ {g_1^{(n-1)}(Y)A_1(X)^{n-1}b_1\over
n!}}\right
)&\equiv 0 \left ( {\text{ mod }b_1^{n}}\right )\\
\Rightarrow \left(\displaystyle{{g_1(Y) b_1^{n-1}\over 1!}+
{g_1^{\prime}(Y)A_1(X) b_1^{n-2} \over 2!}+...+
{g_1^{(n-1)}(Y)A_1(X)^{n-1}\over n!}}\right )&\equiv 0 \left (
{\text{
mod }b_1^{n-1}}\right ) \\
\Rightarrow A_1(X)\left(\displaystyle{{g_1^{\prime}(Y) b_1^{n-2}
\over 2!}+...+ {g_1^{(n-1)}(Y)A_1(X)^{n-2}\over n!}}\right )&\equiv
0 \left ( {\text{
mod }b_1^{n-1}}\right ) \\
\end{align*}
Again since $gcd(A_1(X),b_1)=1$ we get,
\begin{equation}\label{1.3}
\left(\displaystyle{{g_1^{\prime}(Y) b_1^{n-2} \over 2!}+
{g_1^{\prime\prime}(Y) b_1^{n-3}A_1(X) \over 3!}+...+
{g_1^{(n-1)}(Y)A_1(X)^{n-2}\over n!}}\right )\equiv 0 \left (
{\text{ mod }b_1^{n-1}}\right )
\end{equation}
 Putting $X=0$ in \eqref{1.3} we get,
 \begin{equation*}
g_1^{\prime}(Y)\equiv 0 \text{ mod } b_1
 \end{equation*}
 Again the coefficient of
$Y^i$ in $g_1(Y)$ is divisible by $b_1$ for $i\ge 1$ and hence the
coefficient of $Y^{i+1}$ in g(Y) is divisible by $b_1^2$. Repeating
this process we get that the coefficient of $Y^i$ in $g(Y)$ is
divisible by $b_1^i$ for all $i\ge 1$. i.e $g(Y)=D(b_1Y)$for some
$D(Y) \in R[Y]$. Thus we have
\begin{align*}
\notag F&=\Bigl(X+D(b_1Y+A_1(X)),Y+\frac{A_1(X)}{b_1}+\frac{A_2(X+D(b_1Y+A_1(X)))}{b_2}\Bigr)\text{ and}\\
F^{-1}&=(X-D(b_1Y-\frac{b_1A_2(X)}{b_2}),Y-\frac{A_2(X)}{b_2}-\frac{A_1(X-D(b_1Y-\frac{\displaystyle b_1A_2(X)}{\displaystyle b_2}))}{b_1}\Bigr)
\end{align*}
Now we will show that $b_1=b_2$.
\begin{equation}\label{1.4}
F\in {\aut}_2^0(R)\Rightarrow {\displaystyle{A_1(X)\over
b_1}+{A_2(X+D(b_1Y+A_1(X)))\over b_2}} \in R[X,Y] \end{equation}

\begin{equation}\label{1.5}F^{-1} \in {\aut}_2^0(R) \Rightarrow
{\displaystyle{-A_2(X)\over b_2}-{A_1(X-D(b_1Y-\frac{\displaystyle b_1A_2(X)}{\displaystyle b_2}))\over b_1}}
\in R[X,Y]\end{equation}
 Putting $Y=0$ in \eqref{1.4} we get that
$$b_2A_1(X)+ b_1A_2(X+D(A_1(X)))\equiv 0(\text{ mod }(b_1b_2).$$
Since $A_1$ and $b_1$ have no common factors it follows that
$b_2\equiv 0(\text{ mod }b_1)$. Similarly from \eqref{1.5} we get
that $ b_2\equiv 0(\text{ mod }b_1)$. Thus $b_2=cb_1$ for some $c\in
R^*$. Replacing $A_2$ with $\displaystyle {A_2\over c}$ and $b_2$
with $b_1$ the result follows.
\end{proof}
\subsection{A Proof of Theorem \ref{length3isstablytame1}}
We may assume that $F$ is of the form in the hypothesis of Lemma \ref{length3structurelemma1}. So from Lemma \ref{length3structurelemma1} we get that
\begin{equation}\label{length3F}
F=\Bigl(X+D(bY+A_1(X)),Y+\frac{A_1(X)+A_2(X+D(bY+A_1(X)))}{b}\Bigr)
\end{equation}
Taking $1$ for $p_2$ , $b$ for $p_1$, $D(Y)$ for $G_2(Y)$ and
$A_1(X)$ for $G_1(X)$ we see that the first co-ordinate of $F$ is in
$ \mathcal{B}^2(A)$ and hence $ F \in \mathcal{B}V_2^2(A)$. By
theorem 3, first coordinate of $F$ is totally stably tame and
hence $F$ is stably tame. This concludes the proof using Theorem \ref{edo'stheorem}
\section{Another Proof of Theorem \ref{length3isstablytame1}}
We now proceed with some preparations for a self contained proof of Theorem \ref{length3isstablytame1}.
\begin{lemma}\label{length3structurelemma2} We use notations from the Lemma \ref{length3structurelemma1}.
 Let $p$ be an irreducible factor of $b$. Then $p$ divides $D(Y)$ or each of the following polynomials.
\begin{enumerate}
\item $D(Y)-D^\prime(0)Y$
\item $A_1(X)-A_1^\prime(0)X$
\item $A_2(X)-A_2^\prime(0)X$
\end{enumerate}
 \end{lemma}
\begin{proof}
Since $b_1=b_2=b$, from \eqref{1.4}and \eqref{1.5} we get the
following.
\begin{align}\label{3.6}
&\frac{\displaystyle A_1(X)+ A_2(X+D(bY+A_1(X)))}{\displaystyle b}
\in R[X,Y]\\
\label{3.7}&\frac{\displaystyle A_2(X)+
A_1(X-D(bY-A_2(X)))}{\displaystyle b} \in R[X,Y]
\end{align}
Putting $Y=0$ in \eqref{3.6} and \eqref{3.7} we have,
\begin{align}
\label{pdivides1}&p\mid {\displaystyle A_1(X)+ A_2(X+D(A_1(X)))}\text{ and }\\
\label{pdivides2}&p\mid{\displaystyle A_2(X)+ A_1(X-D(-A_2(X)))}
\end{align}

Let $S=\frac{\displaystyle R}{\displaystyle pR}$ and denote the image of $a\in R$ in $S$ by $\overline{a}$. Suppose $p$ does not divide $D(Y)$.
 Let $\overline{A_i(X)}=\sum_{j=1}^{n_i}{\overline{a_{ij}}X^j}$
 for $i=1,2$
 and $\overline{D(Y)}=\sum_{j=1}^{n_3}{\overline{d_j}Y^j},\ \overline{d_{n_3}} \neq \overline{0}$. Since $p$ does not divide $D(Y)$ we may assume that $n_3 \ge 1$. Also since $gcd(A_i,b)=1$ we may further assume that $\overline{a_{n_i}}\neq 0$ for $i=1,2$.

 {\bf Case 1 ($n_2\ge n_1$):-}\\
  Since $p \mid A_1(X)+ A_2(X+D(A_1(X)))$,\\
 \begin{align}\label{*}
 \notag \overline{A_1(X)+ A_2(X+D(A_1(X)))}&=
 \sum_{j=1}^{n_1}{\overline{a_{1j}}X^j}+\sum_{j=1}^{n_2}{\overline{a_{2j}}
 \Bigl(X+\sum_{l=1}^{n_3}{\overline{d_j}(\sum_{m=1}^{n_1}{\overline{a_{1j}}X^m})^l\Bigr)^j}}\\
 &=\overline{ 0 }
 \end{align}
Suppose $n_1=n_3=1$ and $n_2>1$ then the top term in the expression
\eqref{*} is \\
$\overline{a_{2n_2}(1+d_1a_{11})^{n_2}}X^{n_2}=\overline{0}$
which implies $\overline{1+d_1a_{11}}=\overline{0}$. Now lets look at the
lowest degree term in the expression \eqref{*} which is
$$\overline{a_{11}+a_{21}(1+d_1a_{11})}X=\overline{a_{11}}X=
\overline{0}.\ \text{ Hence } \overline{A_1(X)}=\overline{0}.$$
 This is a contradiction to assumption that $gcd(A_1(X),b)=1$. Thus $n_1=n_3=1 \Rightarrow n_2=1$.

 So lets assume that $n_1>1$ or $n_3>1$. We look at the coefficient of the highest degree term in the expression \eqref{*}.

 Suppose $n_1>1$ and $n_3>1$. Then $n_2>1$ and hence $n_1n_2n_3>n_1$.\\
  So the highest degree term in \eqref{*} is
  $\overline{ a_{2n_2}d_{n_3}^{n_2}a_{1n_1}^{n_2n_3}}X^{n_1n_2n_3}=\overline{0}$.\\
This is a
contradiction to the assumption that $\overline{d_{n_3}},\ \overline{a_{1n_1}},\text{ and } \overline{a_{2n_2}}$ are not equal to $\overline{0}.$

Now suppose $n_1>1$ and $n_3=1$. Again $n_2>1$ and hence
$n_1n_2>n_1$.\\
So the highest degree term in the expression \eqref{*} is $
\overline{a_{2n_2}d_1^{n_2}a_{1n_1}^{n_2n_3}}X^{n_1n_2}=\overline{0}$.\\
This is a contradiction to the assumption that $\overline{d_1},\ \overline{a_{1n_1}},\text{ and } \overline{a_{2n_2}}$ are not equal to $\overline{0}.$

Last case is when $n_3>1$ and $n_1=1$.
Again $n_2 \ge 1$ and so $ n_2n_3>n_1$ .\\
So the highest degree term in expression \eqref{*} is
$\overline{a_{n_2}d_{n_3}^{n_2}a_{11}^{n_2n_3}}X^{n_2n_3}=\overline{0}$, again a
contradiction.

Thus $n_1=n_3=1$ which implies $n_2=1$ as well.\\

{\bf Case 2 ($n_1\ge n_2$)}:-\\
Since $ p\mid{\displaystyle A_2(X)+ A_1(X-D(bY-A_2(X)))}$
(from \eqref{pdivides2}) we get the following.\\
\begin{align}\label{**}
 \notag \overline{-A_2(X)- A_1(X-D(-A_2(X)))}&=
 -\sum_{j=1}^{n_1}{\overline{a_{2j}}X^j}-\sum_{j=1}^{n_2}{\overline{a_{1j}}
 \Bigl(X-\sum_{l=1}^{n_3}{\overline{d_j}(\sum_{m=1}^{n_1}{\overline{-a_{1j}}X^m})^l\Bigr)^j}}\\
 &=\overline{ 0 }
 \end{align}
 Proof of Case 2 is exactly like Case 1. We can look at the top term
 of \eqref{**} to conclude that $n_1=n_2=n_3=1$.
\end{proof}

Let $P(X,Y)=D(bY+A_1(X))-D^\prime(0)A_1^\prime(0)X$ and $\bt$ be the
product of irreducible factors of $b$. Then by Lemma
\ref{length3structurelemma2} we have that $\bt \mid P(X,Y)$. So we can rewrite \ref{length3F} as
 \begin{equation*}\label{rewritinglength3F}
  F=(aX+\bt P_1(X,Y), Y+P_2(X,Y)
  \end{equation*}
  where $a =1+D^\prime(0)A_1^\prime(0)$ and
 \begin{align} P_1(X,Y)&={\displaystyle \frac{P(X,Y)}{\bt}}=\displaystyle {D(bY+A_1(X))-(a-1)X \over \bt}\\
  \text{ and } \notag P_2(X,Y)&=\displaystyle\frac{A_1(X)+A_2(X+D(bY+A_1(X)))}{\displaystyle b}
\end{align}

The following lemma was proved in \cite{TameAndWild} when $R=k[t]$. We reprove it here when $R$ is any UFD. The proof given here is simpler.
\begin{lemma}\label{length3structurelemma3} Let $F=F_1^{-1}\circ G \circ F_1  \in \text{L}^{(3)}(R)$ where $F_1=(X,Y+\frac{\displaystyle A_1(X)}{\displaystyle a}),\ G=(X+g(Y),Y),\ A_1(X) \in R[X],\ g(Y) \in K[Y],\ a \in R$. Then $g(Y)=D(aY) \text{ for } \\
D(Y) \in R[Y] \text{ and } a \mid D(Y)$.
\end{lemma}
\begin{proof}
Since $F \in \text{L}^{(3)}(R)$ by Lemma \ref{length3structurelemma1} we have that $g(Y)=D(aY)$. Let $a=a_1^{p_1}a_2^{p_2}\ldots a_l^{p_l}$ where each $a_i$ is irreducible in $R$. Then by Lemma \ref{length3structurelemma2} we know that $a_i\mid D(Y)-D^\prime(0)Y$.
\begin{equation*}
\text{Also, } F=(X+D(aY+A(X)),Y+\frac{\displaystyle A(X)-A(X+D(aY+A(X)))}{\displaystyle a})
\end{equation*}
Putting $Y=0$ in the second coordinate of $F$ we get that
\begin{align}\label{conjugateequations}
\notag A(X)-A(X+D(A(X)))&\equiv 0 \ (\text{ mod }a) \\
 \Rightarrow A(X)-A(X+D(A(X)))&\equiv 0 \ (\text{ mod }a_i^{p_i})  \text{ for every } i
 \end{align}
Similarly putting $Y=0$ in the second coordinate of $F^{-1}$ we get
\begin{align}\label{conjugateequations2}
\notag A(X)-A(X-D(A(X)))&\equiv 0 \ (\text{ mod }a) \\
  \Rightarrow A(X)-A(X-D(A(X)))&\equiv 0 \ (\text{ mod }a_i^{p_i})  \text{ for every } i
\end{align}
From \ref{conjugateequations} and \ref{conjugateequations2} we get that
\begin{equation}\label{conjugateequations3} A(X-D(A(X)) - A(X+ D(A(X)) \equiv 0 \ (\text{ mod }) a_i^{p_i} \text{ for every } i \end{equation}
It is enough to show that for each $i,\ a_i^{p_i} \mid D(Y)$. So we fix an $i$.
\begin{align*}
\text{Let } A(X)&= \sum_{j=1}^n{a_jX^j} \equiv 0 \ (\text{ mod }a_i)  \text{ and }\\
D(Y)&=\sum_{j=1}^m{d_jY^j} \equiv 0 \ (\text{ mod }a_i)
\end{align*}
Looking at the linear part of the left hand side in \ref{conjugateequations3} gives us
\begin{align*}
a_1X-a_1^2d_1X-a_1X-a_1^2d_1X &\equiv 0 \ (\text{ mod }a_i)\\
\Rightarrow 2a_1^2d_1 &\equiv 0 \ (\text{ mod }a_i)\\
\Rightarrow d_1 &\equiv 0 \ (\text{ mod }a_i)\\
\end{align*}
Hence $D(Y) \equiv 0 \ (\text{ mod }a_i)$. Let $D(Y)=a_i^{jt}D_1(Y)$ for some $D_1(Y) \in R[Y]\text{ such that }\\
gcd(D_1(Y),a_i)=1 \text{ and } t\ge 1$. Then \ref{conjugateequations} reads as
\begin{align*}
A(X)-\sum_{j=0}^n{A^{(j)}(X)D_1(A(X))^ja_i^{jt}} &\equiv 0 \ (\text{ mod }a_i^{p_i})\\
\Rightarrow \sum_{j=1}^n{A^{(j)}(X)D_1(A(X))^ja_i^{jt}} &\equiv 0 \ (\text{ mod }a_i^{p_i})
\end{align*}
\
\\

\noindent If $t<p_i$, then we get that $A^\prime(X)D_1(A(X))\equiv 0 \ (\text{ mod }a_i)$. Also, $gcd(A(X),a)=1 \Rightarrow gcd(A^\prime(X),a)=1$.
\begin{align*}
\text{So }D_1(A(X)) &\equiv 0 \ (\text{ mod }a_i)\\
          \sum_{j=0}^{m}{\frac{D_1^{(j)}(0)A(X)^j}{j!}} &\equiv 0 \ (\text{ mod }a_i)\\
\text{Since } D_1(0)=0,\text{ we get that } \sum_{j=1}^{m}{\frac{D_1^{(j)}(0)A(X)^j}{j!}} &\equiv 0 \ (\text{ mod }a_i)\\
          A(X)(\sum_{j=1}^{m}{\frac{D_1^{(j)}(0)A(X)^{j-1}}{j!}}) &\equiv 0 \ (\text{ mod }a_i)\\
         \text{Since }gcd(A(X),a)=1,\ \sum_{j=1}^{m}{\frac{D_1^{(j)}(0)A(X)^{j-1}}{j!}} &\equiv 0 \ (\text{ mod }a_i)\\
       \text{ Putting }X=0 \text{ we get },\ D_1^\prime(0) &\equiv 0 \ (\text{ mod }a_i)
\end{align*}
Proceeding like this we get that $D_1^{(j)}(0) \equiv 0 \ (\text{ mod }a_i)$ which contradicts the fact that $gcd(D_1(Y),a_i)=1$. So $t=p_i$ and we are done.
\end{proof}

\subsection{Proof of the Main Theorem}
Again we may assume that $F$ is of the form in the hypothesis of Lemma \ref{length3structurelemma1}. So we get,
\begin{equation*}
F=\Bigl(X+D(bY+A_1(X)),Y+\frac{A_1(X)+A_2(X+D(bY+A_1(X)))}{b}\Bigr)
\end{equation*}
Let $b=b_1^{s_1}b_2^{s_2}...b_r^{s_r},s_i \ge 1, b_i$ irreducible  in $R$ . We define $s(F)=s_1+s_2+ \ldots + s_r$ and $b(F)=b,\  \bt=b_1b_2 \ldots b_r$. If $b$ is a unit in $R$ then we are done. If not we extend $F$ to $(F,W)\in \L^{(3)}(R[W]))$ and call this extension $F$. Let $\tau=(X,Y,W+P_1(X,Y)),\gamma =(X-\bt W, Y, W).$ Then
\begin{equation}\label{Fandunimodularity}
\gamma\circ F\circ\tau =(aX-\bt W, Y+P_2(X,Y), W+P_1(X,Y))
\end{equation}
 $(a,\bt)$ is a unimodular row and we can extend this to a $3\times 3$ matrix in $\text{SL}_3(R)$ say
 $A= \begin{pmatrix}
a & 0 & -\bt\\
0 & 1 & 0\\
c & 0 & d

\end{pmatrix}$ .
Since $\det(A)=1$, we have
$A^{-1}=\begin{pmatrix}
  d & 0 & \bt \\
  0 & 1 & 0\\
 -c & 0 & a \end{pmatrix}$
 with $ ad+\bt c=1$. \\
 %and since $P_1(X,Y)=\displaystyle {D(bY+A_1(X))-(a-1)X \over \bt}$ we get
 We have $\gamma\circ F\circ\tau\circ A^{-1}=(X,Y+ P_2(dX+\bt W,Y), -cX+aW+P_1(dX+\bt W,Y))$. Substituting for $P_1$ from \ref{rewritinglength3F} we get,
\begin{align*} &\gamma\circ F\circ\tau\circ A^{-1}=\\
                                       &\Biggl(X,Y+P_2(dX+\bt W,Y),-cX+aW+\\
                                       & \hspace{3cm} \frac{\displaystyle D(bY+A_1(dX+\bt W))-(a-1)(dX+\bt W)}{\bt}\Biggr) \\                                                           =&\Bigl(X,Y+P_2(dX+\bt W,Y),W+ \frac{\displaystyle D(bY+A_1(dX+\bt W))+(d-1)X} {\bt} \Bigr) \end{align*}

Notice that $F$ is stable tamely equivalent to $\gamma\circ F\circ\tau\circ A^{-1}.$
For our purpose we may replace $\gamma\circ F\circ\tau\circ A^{-1}$ by
\begin{align*} F^{1} =\Bigl(X,Y+&P_2(dX+\bt W,Y)-P_2(dX,0),\\
                        W+&\frac{\displaystyle D(bY+A_1(dX+\bt W))+(d-1)X - D(A_1(dX))-(d-1)X}{\bt}\Bigr )\\
                     =\Bigl(X,Y+&\displaystyle\frac{A_1(dX+\bt W)+A_2(dX+\bt W+D(bY+A_1(X)))}{\displaystyle b}\\
                     &\hspace{5cm}\frac{-A_1(dX)-A_2(dX+D(A_1(X)))}{\displaystyle b},\\
                     &W+\frac{\displaystyle D(bY+A_1(dX+\bt W))+(d-1)X - D(A_1(dX))-(d-1)X}{\bt} \Bigr)
 \end{align*}

Then $F^1 = F_1^1\circ G^1\circ F^1_1
 \text{ where }$
 \begin{align*}
 F_1^1&=(X,Y+\frac{\displaystyle A^1_1(W)}{b/\bt},W)\\
 G^1&=(X,Y,W+D^1((b/ \bt)Y))\\
 F_2^1&=(X,Y+\frac{\displaystyle A^1_2(W)}{\displaystyle b / \bt},W)\\
 \text{ and }A^1_1(W)&=\frac{\displaystyle A_1(dX+\bt W)-A_1(dX)}{\bt}\\
 A^1_2(W)&=\frac{\displaystyle A_2(dX+\bt W+D(A_1(dX)))-A_2(dX+D(A_1(dX))}{\displaystyle \bt}\\
 D^1(Y)&=\frac{\displaystyle D(bY+A_1(dX))-D(A_1(dX)}{\displaystyle \bt}
 \end{align*}

with $A^1_1,A^1_2 \in R[X][W],D^1 \in R[X][Y]$. Clearly $ b(F^1)=b/\bt$ and hence if $b$ is not a unit in $R$ then $s(F^1)<s(F)$. Then we are done by induction on $s(F)$. Notice that at the next stage of the induction the matrix $A$ appearing in the proof will have entries from $R[X]$. This is why we need the hypothesis that $\text{SL}_2(R[X_1,X_2,\ldots, X_n]) = \text{E}_2(R[X_1,X_2,\ldots, X_n])$ for all $n$.

\section*{Acknowledgements}The author wishes to thank his advisor David Wright for all the guidance and stimulating discussions.
 \bibliographystyle{alpha}
\bibliography{stablytamepaper}

\end{document}